\documentclass[10pt]{amsart}

 \usepackage{amsmath,amsthm}
\usepackage{amssymb,latexsym}

\makeatletter
\@addtoreset{equation}{section}
\makeatother

\newtheorem{theorem}{Theorem}[section]
\newtheorem{proposition}[theorem]{Proposition}
\newtheorem{lemma}[theorem]{Lemma}
\newtheorem{corollary}[theorem]{Corollary}

\newtheorem{conjecture}[theorem]{Conjecture}

\theoremstyle{definition}
\newtheorem{definition}[theorem]{Definition}
\newtheorem{remark}[theorem]{Remark}

\newcommand{\wbar}[1]{\overline{#1}}
\newcommand{\what}[1]{\widehat{#1}}
\newcommand{\wtil}[1]{\widetilde{#1}}
\newcommand{\til}[1]{\tilde{#1}}

\newcommand{\id}{\operatorname{id}}
\newcommand{\ind}{\operatorname{ind}}
\newcommand{\supp}{\operatorname{supp}}
\newcommand{\spn}{\operatorname{span}}

\newcommand{\fA}{\mathcal{A}}
\newcommand{\fB}{\mathcal{B}}
\newcommand{\fC}{\mathcal{C}}

\newcommand{\fH}{\mathcal{H}}
\newcommand{\fK}{\mathcal{K}}
\newcommand{\fM}{\mathcal{M}}
\newcommand{\fS}{\mathcal{S}}

\newcommand{\fV}{\mathcal{V}}

\newcommand{\Cee}{\mathbb{C}}

\newcommand{\Tee}{\mathbb{T}}
\newcommand{\Zee}{\mathbb{Z}}
\newcommand{\Que}{\mathbb{Q}}
\newcommand{\En}{\mathbb{N}}

\newcommand{\alp}{\alpha}
\newcommand{\del}{\delta}
\newcommand{\Del}{\Delta}

\newcommand{\Gam}{\Gamma}
\newcommand{\lam}{\lambda}

\newcommand{\sig}{\sigma}

\newcommand{\norm}[1]{\left\Vert#1\right\Vert}
\newcommand{\cbnorm}[1]{\left\Vert#1\right\Vert_{\mathrm{cb}}}

\newcommand{\ball}{\mathrm{b}}
\newcommand{\cb}{\mathrm{cb}}
\newcommand{\bl}{\mathrm{L}}
\newcommand{\fal}{\mathrm{A}}
\newcommand{\fsal}{\mathrm{B}}
\newcommand{\meas}{\mathrm{M}}

\newcommand{\vn}{\mathrm{VN}}

\begin{document}

\title[Similarity degree of Fourier algebras]
{Similarity degree of Fourier algebras}

\author{Hun Hee Lee, Ebrahim Samei and Nico Spronk}

\begin{abstract}
We show that for a locally compact group $G$, amongst a class which contains amenable and
small invariant neighbourhood groups, that its Fourier algebra $\fal(G)$ satisfies a completely bounded 
version Pisier's similarity property with similarity degree at most $2$.  Specifically, 
any completely bounded homomorphism $\pi:\fal(G)\to\fB(\fH)$
admits an invertible $S$ in $\fB(\fH)$ for which $\|S\|\|S^{-1}\|\leq \cbnorm{\pi}^2$ 
and $S^{-1}\pi(\cdot)S$ extends to a $*$-representation of 
the C*-algebra $\fC_0(G)$.  This significantly improves some results due to Brannan and Samei
({\em J. Funct. Anal.} 259, 2010) and Brannan, Daws and Samei 
({\em M\"{u}nster J. Math} 6, 2013).
We also note that $\fal(G)$ has completely bounded similarity degree $1$ if and only if it is completely isomorphic to an operator algebra if and only if $G$ is finite.
\end{abstract}

\maketitle

\footnote{{\it Date}: \today.


The first named author would like to thank Seoul National University through the Research Resettlement Fund for the new faculty and the Basic Science Research Program through the National Research Foundation of Korea (NRF), grant  NRF-2015R1A2A2A01006882.
The second named author would like to thank NSERC Grant 409364-2015.
The third named author would like thank NSERC Grant 312515-2010.}


In \cite{pisier}, Pisier launched a major assault on a large class of similarity problems
including Kadison's similarity problem for C*-algebras and Dixmier's similarity problem
for groups.  Pisier's partial solution to Dixmier's problem (\cite[Theorem 3.2]{pisier}
for discrete groups, which is adapted in the third named author's thesis \cite{spronkT} to
general locally compact groups) is encapsulated in the following result.

\begin{theorem}\label{theo:pisdix}
Let $G$ be a locally compact group.  Any bounded homomorphism of the group algebra to the 
algebra of bounded operators on a Hilbert space, $\pi:\bl^1(G)\to\fB(\fH)$,
admits an invertible $S$ in $\fB(\fH)$ for which 

{\bf (a)} $\pi_S=S\pi(\cdot)S^{-1}$ is
a $*$-representation:  $\pi_S(f^*)=\pi_S(f)^*$, and with

{\bf (b)} $\|S\|\|S^{-1}\|\leq K\norm{\pi}^2$, for some constant $K$

\noindent if and only if $G$ is amenable.
\end{theorem}

We shall refer to success of (a) as the {\it similarity property} for the algebra $\bl^1(G)$.
The exponent $2$ in (b) means that $\bl^1(G)$ has {\it similarity degree} at most $2$.

Our goal is to gain an analogue of this similarity theorem for the Fourier algebra $\fal(G)$.
This algebra was defined by Eymard~\cite{eymard}, and is the  Pontryagin dual object to $\bl^1(G)$.  
We note that $\fal(G)$ is a Banach algebra of functions on $G$, which is conjugation-closed
and admits spectrum $G$.  In particular, $\fal(G)$ is a dense subalgebra of $\fC_0(G)$.
We succeed in our goal, though only under extra hypotheses on the locally compact 
group $G$, as specified in Theorem \ref{theo:main}, below.

One of the most intriguing aspects of Pisier's methods, is that they add to the convincing body of work
which indicates the 
value of considering the operator space structure to attack certain questions which appear to have
no a priori need for such.  However, as we explain below, we shall require a natural operator space
structure on the Fourier algebra, to frame our problem.

We observe that for abelian $G$, the Fourier transform of the involution on $\bl^1(\what{G})$,
$f\mapsto f^*$, is given on $\fal(G)$ by complex conjugation, $u\mapsto \bar{u}$.  Thus this is the suitable
involution on $\fal(G)$ which we shall consider for any locally compact $G$.   Our goal is to characterize
representations $\pi:\fal(G)\to\fB(\fH)$ which admit $S$ in $\fB(\fH)$ for which $\pi_S=S\pi(\cdot)S^{-1}$
is a $*$-representation:  $\pi_S(\bar{u})=\pi_S(u)^*$; this is the {\it similarity property} for $\fal(G)$.  
Because any such homomorphism $\pi$ must
necessarily factor through the minimal operator space of functions $\fC_0(G)$
(which is the enveloping C*-algebra of the involutive algebra $\fal(G)$ -- see begining of \S \ref{sec:main}),  
$\pi$ must be completely bounded.  Hence we wish to know if all completely bounded homomorphisms
of $\fal(G)$ enjoy the similarity property.  This was the goal of the investigation of Brannan
and the second named author~\cite{brannans} and was followed upon by the same two authors in 
conjunction with M. Daws~\cite{brannands}.  They only gained a complete solution
for small invariant neighbourhood (SIN) groups in the first paper, and
for amenable groups containing open SIN subgroups in the second.   
We gain a complete solution for a class of groups containing both amenable and SIN
groups (Theorem \ref{theo:main}).    In all of their calculations, it was shown that
the completely bounded similarity degree is at most $4$, i.e.\ $\|S\|\|S^{-1}\|\leq\cbnorm{\pi}^4$ in 
Theorem \ref{theo:pisdix} (b), above, with $\fal(G)$ replacing $\bl^1(G)$.  Our Theorem \ref{theo:main}
improves their results by showing that the completely bounded similarity degree is at most $2$, in  cases 
generalizing theirs.  However, we also indicate some cases with
completely bounded similarity degree at most $6$ (Corollary \ref{cor:simsix}); we suspect that
completely bounded similarity degree at most $2$ holds generally. Two ancillary
representations $\check{\pi}$ and $\pi^*$ are introduced in the article \cite{brannans}, which we 
mention in Corollary \ref{cor:brannans}.  Brannan and Samei showed that simultaneous complete 
boundedness for either of the pairs $\pi$ and $\check{\pi}$, or $\pi$ and $\pi^*$, 
characterizes the completely bounded similarity property for $\fal(G)$.
We augment their result with a larger class of groups. 
At the end of the paper we further note that the completely bounded similarity degree is exactly $1$ only if
$G$ is finite.

It is proved by Choi and the second named author~\cite{chois}, that if $G$ contains a discrete copy
of a free group, then there are bounded homomorphisms $\pi:\fal(G)\to\fB(\fH)$ which are not
completely bounded.  In particular, such homomorphisms cannot be similar to $*$-homomorphisms.  
With this in mind we shall restrict our attention to homomorphisms which are assumed to be completely
bounded.  The question of when a bounded homomorphism from $\fal(G)$ to $\fB(\fH)$
is completely bounded is trivial
in the case that $G$ admits an abelian subgroup of finite index, since
then $\fal(G)$ is naturally a maximal operator space (see Forrest and Wood~\cite{forrestw}); 
but for general (amenable) $G$, we have no tools to asses the question.  
We discuss this question briefly in Remark \ref{rem:autobdd},
below.

The main tool for our similarity result stems from the approach of Pisier to general similarity problems.  
We offer a purposely limited summary of it in the next section.  Our main results are stated and proved
in the section which follows.

\section{Pisier's approach to similarity problems}

Let $\fA$ be a Banach algebra with an operator space structure.  In practice, this will usually be a 
completely contractive Banach algebra, i.e.\ the mutliplication map extends to a complete contraction
from the operator projective tensor product $\fA\hat{\otimes}\fA$ to $\fA$.  We shall assume that
{\em there exists at least one injective completely contractive homomorphism $\pi:\fA\to\fB(\fH)$}, for
some Hilbert space $\fH$.

If $\pi:\fA\to\fB(\fH)$ is a completely
bounded homomorphism, then
\begin{equation}\label{eq:normacheived}
\cbnorm{\pi}=\sup\left\{\cbnorm{P_{\fH'}\pi(\cdot)|_{\fH'}}:
\begin{matrix} \fH'\text{ is a }\pi\text{-invariant closed subspace} \\
\text{of }\fH\text{ with }\dim\fH'\leq \mathrm{den}(\fA)\phantom{mmm}\end{matrix}\right\}
\end{equation}
where $\dim\fH'$ denotes the hilbertian dimension, $P_{\fH'}$ is the orthogonal projection onto $\fH'$
and $\mathrm{den}(\fA)$ is the smallest cardinality of a dense subset of $\fA$.
Indeed, for each $n$ and $k$ let $\xi=(\xi_{1k},\dots,\xi_{nk})$ and $\eta=(\eta_{1k},\dots,\eta_{nk})$
be unit vectors in $\fH^n\cong\ell^2_n\otimes^2\fH$ 
for which $\sup\{|\langle\id_n\otimes \pi(a)\xi|\eta\rangle|:\|\id_n\otimes \pi(a)\|\leq 1\}>
1-\frac{1}{k}$, and let $\fH'$ be the smallest space containing the cyclic subspaces
$\pi(\fA)\xi_{nk}$ and $\pi(\fA)\eta_{nk}$ and the vectors $\xi_{nk},\eta_{nk}$, themselves.

We fix $\fH_\fA$ with $\dim\fH_\fA= \mathrm{den}(\fA)$ and for any $c\geq 1$ we let
\[
\Gamma_c(\fA)=\{\pi:\fA\to\fB(\fH_\fA)\;|\;\pi\text{ is a homomorphism with }\cbnorm{\pi}\leq c\}.
\]
Let $\fH_c=\fH^{\Gamma_c(\fA)}$, so $\iota_c=\bigoplus_{\pi\in\Gam_c(\fA)}\pi:\fA\to\fB(\fH_c)$
is an injective homomorphism with $\cbnorm{\iota}\leq c$, and let
$\wtil{\fA}_c=\wbar{\iota_c(\fA)}$.  If $\pi:\fA\to\fB(\fH)$ is a homomorphism with
$\cbnorm{\pi}\leq c$, then (\ref{eq:normacheived}) provides a complete contraction
$\til{\pi}_c:\wtil{\fA}_c\to \fB(\fH)$ with $\til{\pi}_c\circ\iota_c=\pi$.  

The composition of maps on $N$-fold Haagerup tensor products
\[
\fA^{N\otimes^h}\xrightarrow{\iota_c^{N\otimes}}
{\wtil{\fA}_c}^{N\otimes^h}\xrightarrow{\text{multiplication}}\wtil{\fA}_c
\]
has completely bounded norm no larger than $c^N$, thanks to the fact that the Haagerup tensor product
completely contractively linearizes operator multiplication.  Hence the map
\begin{equation}\label{eq:mc}
m_{N,c}:\fA^{N\otimes^h}\to\wtil{\fA}_c,\; m_{N,c}(u_1\otimes \dots\otimes u_N)
=\frac{1}{c^N}\iota_c(u_1)\dots\iota_c(u_N)
\end{equation}
is a complete contraction. 

We can now state, for our context, the easy direction
of \cite[Theorem 2.5]{pisier}.  This is all that we shall require in the sequel.

\begin{theorem}\label{theo:sim}
Let $\fA$ be a Banach algebra and an operator space and $N$ a positive integer.   
Suppose the map $m_{N,1}:\fA^{N\otimes^h}\to \wtil{\fA}_1$ is a complete surjection,
i.e.\ there is a $K>0$ such that for any $n$, 
the unit ball $\ball_1(M_n\otimes\wtil{\fA}_1)$ is contained in the image of the radius 
$K$-ball, $\id_n\otimes m_{N,1}(\ball_K(M_n\otimes\fA^{N\otimes^h}))$.
Then any completely bounded homomorphism $\pi:\fA\to\fB(\fH)$ admits an
invertible operator $S$ in $\fB(\fH)$ for which
\[
\cbnorm{S\pi(\cdot)S^{-1}}\leq 1\text{ and }
\|S\|\|S^{-1}\|\leq K\cbnorm{\pi}^N.
\]
\end{theorem}

\begin{proof}
Let us give a short proof for conveneince of the reader.  
We wish to establish the existence of a completley bounded homomorphism
$\iota_{1,c}:\wtil{\fA}_1\to\wtil{\fA}_c$, which extends $\iota_c:\fA\to\wtil{\fA}_c$.
We observe that  on algebraic tensor products we have the following equality of
(a priori unbounded) operators
\[
\iota_{1,c}\circ m_{N,1}|_{\fA^{N\otimes}}=c^Nm_{N,c}|_{\fA^{N\otimes}}:\fA^{N\otimes}\to\wtil{\fA}_c
\]
where $m_{N,c}$ is defined in (\ref{eq:mc}), above.
Since $m_{N,c}:\fA^{N\otimes^h}\to\wtil{\fA}_c$ is a complete contraction, $\iota_{1,c}\circ m_{N,1}$
extends to a completely bounded map on $\fA^{N\otimes^h}$.  Furthermore, since
$m_{N,1}:\fA^{N\otimes^h}\to \wtil{\fA}_1$ is assumed to be completely surjective,  with matricial $K$-balls
mapped onto matricial unit balls, we  find that
\[
\cbnorm{\iota_{1,c}}\leq K\cbnorm{\iota_{1,c}\circ m_{N,1}}=K\cbnorm{c^Nm_{N,c}}\leq Kc^N.
\]
Thus we have a completely bounded homomorphism $\iota_{1,c}:\wtil{\fA}_1\to\wtil{\fA}_c$, with 
$\iota_{1,c}\circ\iota_1=\iota_c$.

We recall Paulsen's similarity
theorem \cite{paulsen}, that any completely bounded homomorphism from an operator algebra,
$\sig:\fB\to\fB(\fH)$, admits an invertible $S$ in $\fB(\fH)$ for which
$\sig_S=S\sig(\cdot)S^{-1}$ is a complete contraction and $\|S\|\|S^{-1}\|=\cbnorm{\sig}$.
A non-unital version of this may be found in  \cite[Theorem 5.1.2]{blecherlm}.
Hence for any homomorphism $\pi:\fA\to\fB(\fH)$ with $\cbnorm{\pi}\leq c$, the completely bounded
homomorphism $\til{\pi}_c\circ\iota_{1,c}:\wtil{\fA}_1\to\fB(\fH)$ admits an $S$ in $\fB(\fH)$
for which
\[
\cbnorm{S\til{\pi}_c\circ\iota_{1,c}(\cdot)S^{-1}}\leq 1\text{ and }
\|S\|\|S^{-1}\|=\cbnorm{\til{\pi}_c\circ\iota_{1,c}}\leq \cbnorm{\iota_{1,c}}.
\]
In particular, since $(\til{\pi}_c\circ\iota_{1,c})\circ\iota_1=\til{\pi}_c\circ\iota_c=\pi:\fA\to\fB(\fH)$
and $\iota_1$ is a complete contraction, we see that
\[
\cbnorm{S\pi(\cdot)S^{-1}}\leq 1\text{ and }\|S\|\|S^{-1}\|\leq\cbnorm{\iota_{1,c}}\leq Kc^N.
\]
The choice of $c=\cbnorm{\pi}$ gives the desired inequalities.
\end{proof}

We shall take the liberty to refer to the smallest $N$, satisfying the hypotheses of the theorem above,
the {\it completely bounded similarity degree}, $d_{\cb}(\fA)$ of the  Banach 
algebra and operator space $\fA$.  This differs mildly
from Pisier's definition in \cite{pisier}, though coincides in cases where $\fA$ has a bounded 
approximate identity, for example.  

We notice that $d_{\cb}(\fA)=1$ entails that 
$\iota_1:\fA\to\wtil{\fA}_1$ is a completely contractive complete isomorphism.
Hence the conclusion of Theorem \ref{theo:sim} is simply an expression of Paulsen's similarity
theorem.

\begin{remark}
Since contractive representations of a locally compact group
$G$ on $\fB(\fH)$ is necessarily unitary, and such representations
are in bijective correspondence with homomorphisms of $\bl^1(G)$ with non-degenerate range, 
we have $\wtil{\bl^1}(G)_1\cong\mathrm{C}^*(G)$, where $\bl^1(G)$ is equipped with the
maximal operator space structure.
The third named author, \cite[Theorem 6.10]{spronk}, showed that $m_{2,1}:\bl^1(G)\otimes^h\bl^1(G)\to
\mathrm{C}^*(G)$ is a complete quotient
map when $G$ is amenable.  Hence a proof of the easy direction of Theorem \ref{theo:pisdix}, with 
constant $K=1$, follows from Theorem \ref{theo:sim}.  Of course the original proof of Dixmier 
\cite{dixmier} is more elementary.

We note that $d_{\cb}(\bl^1(G))=1$ only when
$G$ is finite.  Indeed, $\bl^1(G)$ is Arens regular exactly when $G$ is finite, thanks to Young \cite{young};
and an operator algebra is always Arens regular, thanks to Civin and Yood \cite{civiny}.
\end{remark}

\section{Similarity for Fourier algebras}\label{sec:main}

Let $G$ be a locally compact group.
As alluded to in the introduction, $\fC_0(G)$ is the enveloping C*-algebra of $\fal(G)$.
Indeed, any C*-algebra generated by a $*$-homomorphism of $\fal(G)$ admits spectrum which is a 
subset of $G$.  To use the result of the prior section we need to identify, for certain locally compact groups, 
the universal
algebra generated by completely contractive homomorphisms of their Fourier algebras.
The operator space structure on $\fal(G)$ is always the one arising from the duality relation
$\fal(G)^*\cong\vn(G)$; see \cite[\S 3.2]{effrosrB}.

\begin{proposition}\label{prop:contfal}
Let $G$ be a locally compact group for which the connected component of the identity $G_e$
is amenable.  Then

{\bf (a)} every completely contractive homomorphism $\pi:\fal(G)\to\fB(\fH)$ is a $*$-homomorphism.

In particular, we have that

{\bf (b)} the inclusion map $j:\fal(G)\to\fC_0(G)$ extends to a complete isometry
$\til{\jmath}_1:\wtil{\fal}(G)_1\to\fC_0(G)$.

Moreover, property (b) implies property (a).
\end{proposition}

\begin{proof}
If $G$ is amenable, then (a) is the result \cite[Theorem 8.1]{brannands}.   Now if  $G_e$ is amenable.
It is well known that $G/G_e$ contains a compact open subgroup, which is hence the image of an open subgroup $H$ of $G$ which is an amenable extension of a compact group, thus itself amenable. 
Any completely contractive representation $\pi$ restricts to a completely contractive
representation $\pi|_{\fal(H)}:\fal(H)\to\fB(\fH)$, which is hence a $*$-representation.
Thus it follows from \cite[Lemma 24]{brannans} that $\pi$ itself must be a $*$-representation.
Hence (a) holds in this case.

Given a completely contractive homomorphism $\pi$, if (a) holds then
$\wbar{\pi(\fal(G))}$ is a commutative C*-algebra, hence a quotient
of the enveloping C*-algebra $\fC_0(G)$.  In particular, $\wtil{\fal}(G)_1=\wbar{\iota_1(\fal(G))}$
enjoys this property.  Hence, the completely contractive inclusion homomorphism
$j$ extends to a homomorphism of C*-algebras $\til{\jmath}_1$.
Since $\fal(G)$ has spectrum $G$, so too must $\wtil{\fal}(G)_1$, and thus
$\ker\til{\jmath}_1$ is the radical of $\wtil{\fal}(G)_1$, hence $\{0\}$.
Thus  $\til{\jmath}_1$ is an isomorphism of C*-algebras, whence a complete isometry.

If (b) holds, then $\iota_1={\til{\jmath}_1}^{-1}\circ j$, and hence is a $*$-homomorphism.
It is a standard exercise that any contractive homomorphism from a C*-algebra to $\fB(\fH)$
is a $*$-homomorphism (indeed one cuts down to the non-degenerate situation, extends
to the unitization and observes that unitaries are sent to unitaries).  Hence for any
completely contractive homomorphism $\pi:\fal(G)\to\fB(\fH)$, the extension $\til{\pi}_1:
\wtil{\fal}(G)_1\to\fB(\fH)$ is a $*$-homomorphism, thus so too is $\pi=\til{\pi}_1\circ\iota_1$.
\end{proof}

\begin{remark}\label{rem:contfal}
{\bf (i)}  If $G$ is totally disconnected, then the above result is elementary.  Indeed,
there is a neighbourhood base $\{H_\beta:\beta\in B\}$
for the topology of $G$ consisting of open subgroups.  Since $\spn\{1_{sH_\beta}:s\in G,\beta\in B\}$
is dense in $\fal(G)$, and each $\iota_1(1_{sH_\beta})$ is a contractive idempotent in 
an operator algebra, it follows that $\wtil{\fal}(G)_1$ is a C*-algebra, whence isomorphic to $\fC_0(G)$.

{\bf (ii)}  If the connected component of the identity of $G$ is a small invariant neighbourhood group,
i.e.\ a SIN group, then it is shown in \cite[\S 6]{brannans} that any completely contractive
representation $\pi:\fal(G)\to\fB(\fH)$ is automatically a $*$-representation.  
Notice that a connected SIN group is the direct product of a vector group and a compact group thanks to
the Freudenthal-Weil theorem (\cite[16.4.6]{dixmierB}), and, in particular, is amenable.
\end{remark}

We recall that bounded measures act on $\bl^1(G)$ by convolution.  In particular, for a Dirac measure
$\del_x$ and $v$ in $\bl^1(G)$ we have for a.e.\ $y$ in $G$ that $\del_x\ast v(y)=v(x^{-1}y)$ and
$v\ast\del_x=\frac{1}{\Del(x)}v(yx^{-1})$, where $\Del$ is the Haar modular function.

\begin{definition} A group has the {\it quasi-small invariant neighbourhood property} 
(in short, is QSIN) provided
that $\bl^1(G)$ admits a quasi-central bounded approximate identity, i.e.\ a bounded approximate identity
$(v_\alp)$ for which
\begin{equation}\label{eq:qsin}
\lim_\alp\norm{\del_x\ast v_\alp\ast\del_{x^{-1}}-v_\alp}_1=0
\text{ uniformly for }x\text{ in compact subsets of }G.
\end{equation}
\end{definition}

\begin{remark}\label{rem:scope}
{\bf (i)}  The QSIN property implies that $\bl^\infty(G)$ admits a conjugation invariant 
mean which resticts on continuous functions to evaluation at $e$; see Losert and Rindler \cite{losertr}.
The property of admitting a conjugation invariant mean is called {\it inner amenability};
this is not to be confused with the terminology in much of the literature in which a discrete group
is called inner amenable if $\ell^\infty(G)$ admits a non-trivial (i.e. not extending the evaluation at $e$) 
conjugation invariant mean.

Any amenable or SIN group is QSIN; see \cite{losertr} and Mosak \cite{mosak}, respectively.
The connected component of the identity of a QSIN group must be amenable, as is  
shown in \cite{losertr}.  In particular, a connected group enjoys the QSIN property if and only
if it is amenable.  

{\bf (ii)}  A result of Lau and Paterson (\cite[Corollary 3.2]{laup})
tells us that {\it $G$ is amenable if and only if $\vn(G)$ is injective and $G$ is inner amenable}.
Consider a non-compact reductive group over the $p$-adic numbers, $\mathrm{G}(\Que_p)$, 
which is non-amenable and totally disconnected.  
Any such group is a type I group by a result of Bern\v{s}te\u{\i}n \cite{bernshtein}, 
and hence admits an injective von Neumann algebra (see the survey paper \cite{paterson}).  
Hence $\mathrm{G}(\Que_p)$ provides a totally disconnected non-inner amenable, hence
non-QSIN, group.

{\bf (iii)} The existence of an open normal QSIN (even compact abelian) 
subgroup is not sufficient for a group to be QSIN.  We note two classes of examples
shown in \cite{losertr} of non-QSIN groups. 
For each $n>1$, consider $\Tee^n\rtimes\mathrm{SL}_n(\Zee)$.
Also, if $\Gamma$ is a non-amenable discrete group
and $K$ any compact group, consider the wreath product $K^\Gamma\rtimes \Gamma$.
In each example, the normalized Haar measure on $\Tee^n$, respectively $K^\Gamma$, provides a 
conjugation-invariant mean for the group, hence each group is inner amenable.

{\bf (iv)} Let us note another class of  non-QSIN, examples.  
Let $K$ be any local field and $\Gam$ be a subgroup of $\mathrm{GL}_2(K)$, 
containing a countable subgroup $\Gam_0$ whose closure $\wbar{\Gam_0}$ in 
$\mathrm{GL}_2(K)$ admits a non-solvable open subgroup.
Consider $G=K^2\rtimes\Gam$, where $\Gam$ is treated as a discrete group.
If there were a net $(u_\alp)$ satisfying (\ref{eq:qsin}), we may suppose its elements are supported in the open subgroup $K^2$; see (\ref{eq:stokke}), below.   Any cluster point of this net in $\bl^\infty(K^2)^*$
would restrict to an $\Gam$-invariant mean.  Now $K^2\setminus\{0\}
\cong \mathrm{GL}_2(K)/S$, where $S$ is the stabilizer subgroup of some point in 
$K^2\setminus\{0\}$, which is solvable, hence amenable.  
The Connes-Sullivan-Zimmer result of Breuillard and Gelander
\cite[1.10]{breuillardg} tells us that $\mathrm{SL}_2(K)/S$ is not an amenable $\Gam_0$-space,
and hence $\bl^\infty(K^2)$ cannot admit a $\Gam$-invariant mean.  
\end{remark}

The following is the main result of this note.

\begin{theorem}\label{theo:main}
Let $G$ be a QSIN group.
Then for any completely bounded homomorphism $\pi:\fal(G)\to\fB(\fH)$
there is an invertible $S$ in $\fB(\fH)$ for which 

{\bf (a)} $\pi_S=S\pi(\cdot)S^{-1}$ is a 
$*$-representation of $\fal(G)$:  $\pi_S(\bar{u})=\pi_S(u)^*$; and 

{\bf (b)} $\|S\|\|S^{-1}\|\leq \cbnorm{\pi}^2$.
\end{theorem}

\begin{remark}\label{rem:improve}
For a QSIN group, we have that $\fal(G)$
satisfies the similarity property, with completely bounded similarity degree $d_{\cb}(\fal(G))\leq 2$.  
Hence our result significantly improves upon \cite[Theorem 20]{brannans} where,  for SIN groups, 
the estimate $d_{\cb}(\fal(G))\leq 4$ is achieved; and upon \cite[Theorem 8.2]{brannands}
where the same is shown for any amenable group admitting an open SIN-subgroup.
\end{remark}

\begin{proof}
The QSIN property implies that the connected component of $G$ is amenable, as noted
in Remark \ref{rem:scope} (i), above.
Thanks to Theorem \ref{theo:sim} and the fact that $\wtil{\fal}(G)_1\cong\fC_0(G)$ for any group 
whose connected component of the identity is amenable as shown in Proposition \ref{prop:contfal}, it suffices
to show that 
\begin{equation}\label{eq:multcs}
m=m_{2,1}:\fal(G)\otimes^h\fal(G)\to\fC_0(G)\text{ is a complete quotient map.}
\end{equation}
It is easy to see that the adjoint of this map, $m^*:\meas(G)\to\vn(G)\otimes^{eh}\vn(G)\subset
\fC\fB^\sig(\fB(\bl^2(G))$ -- see the article of Blecher and Smith \cite{blechers} --
is given by 
\[
m^*(\mu)=\int_G \lam(s)\otimes\lam(s)\,d\mu(s)
\]
i.e.\ $\langle m^*(\mu)T\xi|\eta\rangle=\int_G\langle \lam(s)T\lam(s)\xi|\eta\rangle$
for $T$ in $\fB(\bl^2(G))$ and $\xi,\eta$ in $\bl^2(G)$.
Hence by duality -- see, for example, \cite[4.1.9]{effrosrB} -- statement (\ref{eq:multcs}) will hold once
we verify that
\begin{equation}\label{eq:multcsd}
\text{the  map }m^*\text{ is completely isometric.}
\end{equation}
We observe that $m^*$ is clearly contractive, and hence completely contractive as $\meas(G)$
is a maximal operator space.  Hence it suffices to exhibit a subspace $\fS$ of $\fB(\bl^2(G))$
for which $m^*:\meas(G)\to\fC\fB(\fS,\fB(\bl^2(G)))$ is a complete isometry, to see that 
$m^*$ itself is a complete isometry.

To prove (\ref{eq:multcsd}), we shall require a particular net of  unit vectors in $\bl^2(G)$, 
also arising from the QSIN assumption.   Thanks to 
Stokke \cite[Theorem 2.6]{stokke}, any QSIN group admits a net $(v_\alp)$ in $\bl^1(G)$ satisfying
(\ref{eq:qsin}) and also the conditions
\begin{equation}\label{eq:stokke}
 v_\alp\geq 0\text{ for each }\alp\text{, and } 
 \supp v_\alp\searrow\{e\}
\end{equation}
where by the latter condition, we mean that given any neighbourhood of the identity,
$\supp v_\alp$ is eventually contained in that neighbourhood.  We may normalize so that
$\int_G v_\alp=1$ for each $\alp$.  We set $u_\alp=\frac{1}{2}(v_\alp+v_\alp^*)$, so that 
\[
u_\alp=u_\alp^*,\,u_\alp\geq 0\text{ and }\int_G u_\alp=1\text{ for each }\alp\text{, and } 
\supp u_\alp\searrow\{e\}.  
\]
Now let
\[
\xi_\alp=u_\alp^{1/2}\text{ in }\bl^2(G).
\]
It is clear that each $\xi_\alp$ is a unit vector.  Let $\lam,\rho:G\to\fB(\bl^2(G)$ denote the left and right 
regular representations.  Using the inequality for positive scalars $|a-b|^2\leq |a^2-b^2|$, we see that
\begin{align}\label{eq:twoqsin}
&{\norm{\lam(x)\rho(x)\xi_\alp-\xi_\alp}_2}^2
=\int_G |\sqrt{\Delta(x)}\xi_\alp(x^{-1}sx)-\xi_\alp(s)|^2\,ds \notag \\
&\qquad\leq \int_G |\Delta(x)u_\alp(x^{-1}sx)-u_\alp(s)|\,ds
=\norm{\del_x\ast u_\alp\ast\del_{x^{-1}}-u_\alp}_1
\end{align}
which tends to $0$ in $\alp$, uniformly for $x$ in compact subsets of $G$.  

We let $U$ denote the unitary on $\bl^2(G)$ given by $U\xi=\frac{1}{\Delta^{1/2}}\check{\xi}$, i.e.\
$U\xi(s)=\frac{1}{\sqrt{\Delta(s)}}\xi(s^{-1})$.  Notice that $U^*=U$ and
$U$ intertwines the right and left regular 
representations:  $U\lam(t)=\rho(t)U$.  Also, for the vectors $\xi_\alp$, above, $U\xi_\alp=\xi_\alp$.
If $f\in\fC_0(G)$, we let $M(f)$ denote the operator of multiplication by $f$ on $\bl^2(G)$.  It is standard to compute that
\begin{gather*}
\lam(t)UM(f)\lam(t)\xi(s)=f(s^{-1}t)\sqrt{\frac{\Delta(t)}{\Delta(s)}}\xi(t^{-1}s^{-1}t),\\
\text{i.e.\ }
\lam(t)UM(f)\lam(t)=M_{\del_t\ast\check{f}}U\lam(t)\rho(t).
\end{gather*}
Hence it follows that
\begin{align}\label{eq:stuff}
\langle m^*(\mu)(UM(f))\xi_\alp|\xi_\alp\rangle
&=\int_G \int_G f(s^{-1}t)\sqrt{\frac{\Delta(t)}{\Delta(s)}}\xi_\alp(t^{-1}s^{-1}t)\wbar{\xi_\alp(s)}\,d\mu(t)\,ds 
\notag \\
&=\int_G \int_G f(s^{-1}t)\sqrt{\frac{\Delta(t)}{\Delta(s)}}\xi_\alp(t^{-1}s^{-1}t)\xi_\alp(s)\,ds\,d\mu(t) \\
&=\int_G\left[\int_G \del_t\ast\check{f} [\lam(t)\rho(t)\xi_\alp]\xi_\alp\right]\,d\mu(t) \notag
\end{align}
where we have used Fubini's theorem in the second equation.  Let us write
\[
w_{t,\alp}=[\lam(t)\rho(t)\xi_\alp]\xi_\alp\text{ in }\bl^1(G).
\]
We have that
\begin{align*}
\left|\int_G (\del_t\ast\check{f}) w_{t,\alp}-f(t)\right|
&=\left|\int_G \left[(\del_t\ast\check{f}) w_{t,\alp}-f(t)u_\alp\right]\right| \\
&\leq \int_G \Bigl[|\del_t\ast\check{f}||w_{t,\alp}-u_\alp|+|\del_t\ast\check{f}-f(t)|u_\alp\Bigr] \\
&\leq \norm{f}_\infty  \int_G |\lam(t)\rho(t)\xi_\alp-\xi_\alp|\xi_\alp + 
\sup_{s\in \supp u_\alp}|f(s^{-1}t)-f(t)| 
\end{align*}
Thanks to Cauchy-Schwarz inequality and (\ref{eq:twoqsin}), as well as uniform continuity of $f$
and (\ref{eq:stokke}), this quantity tends to zero in $\alp$, uniformly for $t$ in compact subsets of $G$.
It other words
\[
\lim_\alp\int_G \del_t\ast\check{f} [\lam(t)\rho(t)\xi_\alp]\xi_\alp=f(t)
\]
uniformly for $t$ in compact subsets of $G$, and, since this net is uniformly bounded, 
the quantities of (\ref{eq:stuff})
tend in $\alp$ to $\int_G f\,d\mu$.  This proves that the map 
$m^*:\meas(G)\to\fB(UM(\fC_0(G)),\fB(\bl^2(G)))$ is an isometry.

Let us now increase dimension.  Given $\mu$ in $M_n\otimes\meas(G)$ and
$f$ in $M_m\otimes\fC_0(G)$ we consider arbitrary unit vectors $\xi$ and $\eta$ in 
$\ell^2_n\otimes^2\ell^2_m$
($nm$-dimensional Hilbert space).  As above we have that
\begin{align*}
\lim_\alp\langle \id_n\otimes&\id_m\otimes m^*(\mu)[\id_m\otimes UM(f)]\xi\otimes\xi_\alp
|\eta\otimes\xi_\alp\rangle \\
&=\lim_\alp\left.\left\langle \id_m\otimes\id_n\left(\int_G (\del_t\ast\check{f}) w_{t,\alp} \otimes d\mu\right)
F\xi\right|F\eta\right\rangle \\
&=\left.\left\langle \id_m\otimes\id_n\left(\int_G f \otimes d\mu\right)F\xi
\right|F\eta\right\rangle
\end{align*}
where $F:\ell^2_n\otimes^2\ell^2_m\to\ell^2_m\otimes^2\ell^2_n$ is the flip map, and, in the doubly indexed
matrix notation of Effros and Ruan \cite{effrosrB}, $\int_G f \otimes d\mu=\left[\int_G f_{kl}\,d\mu_{ij}\right]$
in $M_{mn}$.
Thus $\id_n\otimes m^*:M_n\otimes\meas(G)\to\fB(M_m\otimes UM(\fC_0(G)),M_n\otimes M_m
\otimes\fB(\bl^2(G)))$ is an isometry for each $n,m$.  This shows that $m^*:\meas(G)\to
\fC\fB(UM(\fC_0(G)),\fB(\bl^2(G)))$ is a complete isometry.
\end{proof}

\begin{remark}\label{rem:relations}
{\bf (i)} The relation (\ref{eq:multcsd}) stands in curious comparison with a result of 
Neufang, Ruan and the third named author \cite{neufangrs}, which is equivalent to having that
\[
\mu\mapsto\int_G\lam(s)\otimes\lam(s^{-1})\,d\mu:\meas(G)\to\vn(G)\otimes^{eh}\vn(G)
\]
is a complete isometry for any locally compact group.  The unique isometric linear
weak*-continuous extension of
$\lam(s)\mapsto\lam(s^{-1})$ on $\vn(G)$, which is the adjoint of $u\mapsto\check{u}$ on $\fal(G)$,
is completely bounded only when $G$ is virtually abelian, thanks to an observation of
Forrest and Runde \cite[Proposition 1.5]{forrestr}.  In particular, we know of no means of deducing
our Theorem \ref{theo:main} from this result.

{\bf (ii)}  If $G$ is amenable, we may gain Theorem \ref{theo:main} (a)
from the Theorem of Ruan \cite{ruan}, showing that $\fal(G)$ is operator amenable
(which was used in \cite{brannands}, and hence implicitly in the proof of
Proposition \ref{prop:contfal}, above) and
similarity theorem of Marcoux and Popov \cite{marcouxp}, for commutative
operator amenable operator algebras.  Marcoux and Popov use the
stronger assumption of amenability, instead of operator amenability, in their statements.
Hence we give a sketch of proof.

First, thanks to \cite{ruan},
$\fal(G)$ is operator amenable.   
Hence the commutative operator algebra $\fA=\wbar{\pi(\fal(G))}$ is operator 
amenable and admits a bounded approximate diagonal $(W_\alp)$ in its operator projective
tensor product $\fA\hat{\otimes}\fA$.  Let $\rho:\fA\to\fB(\fK)$
be any non-degenerate completely bounded homomorphism.  We hence consider $\fB(\fK)$ as a
$\fA\hat{\otimes}\fA$-module via $(A\otimes B)\cdot S=\rho(A)S\rho(B)$.
Let $\fM$ be a closed $\rho$-invariant subspace of $\fK$ with orthogonal projection $P$.
It is easy to check that any cluster point $E$ of the net $(W_\alp\cdot P)$ is an idempotent
in $\fB(\fK)$ with range $\fM$ and is an $\fA$-module map:  $E\rho(A)\xi=\rho(A)E\xi$
for $A$ in $\fA$ and $\xi$ in $\fK$.  (This is Helemski\u{\i}'s splitting technique;
see \cite{helemskii,curtisl}.)  Hence $\fA$ satisfies what we might call the
{\it completely bounded total reduction property}.  In particular, for each $A$ in $\fA$, 
the graph $\fM_A=\{(\xi,A\xi):\xi\in\fH\}\subset\fH^2$ is invariant under the 
completely bounded homomorphism $B\mapsto (B\xi,AB\xi)$ of $\fA$ on $\fM_A$.  
The reducibility of modules $\fM_A$ is what is required to give the proof of the main
result of \cite{marcouxp}, that there is a $T$ in $\fB(\fH)$ for which
$T\fA T^{-1}$ is a C*-algebra.  Hence $\pi_T=T\pi(\cdot)T^{-1}:\fal(G)\to T\fA T^{-1}$ is a 
$*$-homomorphism.  Indeed, if $\chi$ is a character on $T\fA T^{-1}$, then $\pi_T^*(\chi):\fal(G)\to\Cee$
is given by pointwise evaluation at an an element $x$ in $G$.  Hence, for $u$ in $\fal(G)$,
$\chi\circ\pi_T(\bar{u})=\bar{u}(x)=\wbar{u(x)}=\wbar{\chi\circ\pi_T(u)}=\chi(\pi_T(u)^*)$.  As this holds
for every character, we see that $\pi_T(\bar{u})=\pi_T(u)^*$.

This technique does not give us obvious means to determine if $d_\cb(\fal(G))\leq 2$.
\end{remark}

We wish to give a modest partial extension of  Theorem \ref{theo:main}.   
The following lemma will be required for this.  We shall use, for an operator space $\fV$,
the spaces $T_n(\fV)$ of Effros and Ruan \cite{effrosrB}.  We shall not require their exact
definition, but shall require the duality formulas
$T_n(\fV)^*\cong M_n\otimes\fV^*$ and $(M_n\otimes\fV)^*\cong T_n(\fV^*)$ (\cite[4.1.1]{effrosrB})
and the fact that each $T_n(\fV_0)\subset T_n(\fV)$ isometrically, if $\fV_0$ is a subspace of $\fV$
(\cite[4.1.8 (ii)]{effrosrB}).  We shall also require the {\it reduced Fourier-Stieltjes algebra}
$\fsal_r(G)$.  This is the algebra of functions on $G$ which consists of
all matrix coefficients of all representations weakly contained in the left regular representation.
Further, it may be identified with the dual of the reduced C*-algebra $\mathrm{C}^*_r(G)$, 
and satisfies that
$\fal(G)$ is $\sig(\fsal_r(G),\mathrm{C}^*_r(G))$-dense in $\fsal_r(G)$ (\cite[(2.16) \& (3.4)]{eymard}).

\begin{lemma}\label{lem:extrfs}
Let $\pi:\fal(G)\to\fB(\fH)$ be a non-degenerate completely bounded homomorphism.  Then there
is a homomorphism $\bar{\pi}:\fsal_r(G)\to\fB(\fH)$ such that
$\bar{\pi}|_{\fal(G)}=\pi$ and $\cbnorm{\bar{\pi}}=\cbnorm{\pi}$.
\end{lemma}

\begin{proof}
The structure of this proof is based on some ideas of Lau and Losert in \cite[\S 4]{laul}.
Let us first show that for any $n$, any $u$ in $M_n\otimes\fsal_r(G)$
may be $\sig(M_n\otimes\fsal_r(G),T_n(\mathrm{C}^*_r(G)))$-approximated by a 
net $(u_\alp)$ from $M_n\otimes\fal(G)$ for which each
$\norm{u_\alp}_{M_n\otimes\fal}=\norm{u}_{M_n\otimes\fsal_r}$.
Indeed, The Hahn-Banach theorem provides  
$\til{u}$ in $M_n\otimes(\vn(G)^*)\cong T_n(\vn(G))^*$ which
resticts to $u$ on $T_n(\mathrm{C}^*_r(G))$.  Then Goldstine's theorem
provides a net $(u_\alp')$ from $M_n\otimes\fal(G)$, which 
$\sig(M_n\otimes(\vn(G)^*),T_n(\vn(G)))$-approaches $u$, and which satisfies
$\norm{u_\alp'}_{M_n\otimes\fal}\leq\norm{u}_{M_n\otimes\fsal_r}$ for each $\alp$.
Let each $u_\alp=\frac{\norm{u}_{M_n\otimes\fsal_r}}{\norm{u_\alp'}_{M_n\otimes\fal}}u_\alp'$.

Now suppose $n=1$.  We let $u\in\fsal_r(G)$ and $(u_\alp)$ be the net from
$\fal(G)$ which is promised above.  A theorem of Granirer and Leinert \cite{granirerl}
tells us that $\lim_\alp\norm{u_\alp v-uv}_\fal=0$ for each $v$ in $\fal(G)$.
The net $(\pi(u_\alp))$, being bounded in $\fB(\fH)$, admits a weak* cluster point $\bar{\pi}(u)$, so
$\bar{\pi}(u)=\text{w*-}\lim_\beta\pi(u_{\alp(\beta)})$ for some subnet
$(u_{\alp(\beta)})$.  Then for $v$ in $\fal(G)$ we have 
$\bar{\pi}(u)\pi(v)=\text{w*-}\lim_\beta\pi(u_{\alp(\beta)}v)=\pi(uv)$, which, by nondegeneracy of $\pi$
tells us that $\bar{\pi}(u)$ is the unique cluster point of $(\pi(u_\alp))$.  Furthermore,
if also $u'\in\fsal_r(G)$, then we likwise see that $\bar{\pi}(u')\bar{\pi}(u)\pi(v)=
\bar{\pi}(u')\pi(uv)=\pi(u'uv)$, so $\bar{\pi}:\fsal_r(G)\to\fB(\fH)$ is a homomorphism.

To compute the norm, we let $n$ be arbitrary, and notice  that
for $u$ and $(u_\alp)$ as in the first paragraph, we have
\[
\norm{\id_n\otimes \bar{\pi}(u)}_{M_n\otimes\fB(\fH)}
\leq \limsup_\alp \norm{\id_n\otimes \bar{\pi}(u_\alp)}_{M_n\otimes\fB(\fH)} 
\leq \cbnorm{\pi}\norm{u}_{M_n\otimes\fsal_r}
\]
so $\cbnorm{\bar{\pi}}\leq\cbnorm{\pi}$.  The converse inequality is trivial.
\end{proof}

The class of groups
in the hypotheses below covers some non-QSIN groups, for example the groups mentioned in 
Remark \ref{rem:scope} (iii) and (iv).  Normality is essential to this particular proof, so we do not capture
any of the groups of Remark \ref{rem:scope} (ii).

\begin{corollary}\label{cor:simsix}
Suppose $G$ admits an amenable open normal subgroup $H$.  
Then for any completely bounded homomorphism $\pi:\fal(G)\to\fB(\fH)$
there is an invertible $T$ in $\fB(\fH)$ for which 

{\bf (a)} $\pi_T=T\pi(\cdot)T^{-1}$ is a 
$*$-representation of $\fal(G)$; and 

{\bf (b)} $\|T\|\|T^{-1}\|\leq \cbnorm{\pi}^6$.
\end{corollary}

\begin{proof}
We have weak contaiment relation on $H$ given by $1_H\prec \lam_H$, thanks to amenability
of $H$ and Hulanicki's theorem \cite{hulanicki}.  Hence by Fell's continuity of induction \cite{fell}
we have on $G$, weak contaiment $\lam_{G/H}=\ind_H^G 1_H\prec \ind_H^G\lam_H=\lam_G$
where $\lam_{G/H}:G\to \fB(\ell^2(G/H))$ denotes the quasi-regular representation.
Hence the Fourier algebra $\fal(G/H)$, which may be naturally identified with the space
of matrix coefficients of $\lam_{G/H}$, is a subalgebra of $\fsal_r(G)$.
Hence we may compress $\pi$ to $\fH'=\wbar{\pi(\fal(G))\fH}$, and
Lemma \ref{lem:extrfs} provides a completely bounded homomorphism
$\bar{\pi}:\fal(G/H)\to\fB(\fH')\subseteq\fB(\fH)$ such that $\bar{\pi}(u)\pi(v)=\pi(uv)$ for
$u$ in $\fal(G/H)$ and $v$ in $\fal(G)$.

Since $G/H$ is discrete, $\bar{\pi}:\fal(G/H)\to\fB(\fH)$
admits an $S$ in $\fB(\fH)$ for which
\[
S\bar{\pi}(\cdot)S^{-1}\text{ is a }*\text{-representation, and }
\|S\|\|S^{-1}\|\leq \cbnorm{\bar{\pi}}^2\leq\cbnorm{\pi}^2.
\]
In particular, if we let for each coset $sH$ in $G/H$, $E_{sH}=\pi(1_{sH})$, then
$P_{sH}=SE_{sH}S^{-1}$ is a projection.  We assume that $\pi$ is non-degenerate, so
$\sum_{sH\in G/H}P_{sH}=I$, in the weak operator topology.  Hence
$\sum_{sH\in G/H}E_{sH}=\sum_{sH\in G/H}S^{-1}P_{sH}S=I$, too.
For each $sH$ in $G/H$ we have that $1_{sH}\fal(G)\cong\fal(H)$, completely isomorphically.
Hence $\pi_{sH}=\pi|_{1_{sH}\fal(G)}:1_{sH}\fal(G)\to\fB(\fH)$ admits an element $S_{sH}$ in
$\fB(\fH)$ for which
\[
S_{sH}\pi|_{sH}(\cdot)S_{sH}^{-1}\text{ is a }*\text{-representation, and }
\|S_{sH}\|\|S_{sH}^{-1}\|\leq \cbnorm{\pi}^2.
\]
We may hence arrange that $\|S_{sH}\|\leq \cbnorm{\pi}$ and $\|S_{sH}^{-1}\|\leq \cbnorm{\pi}$.
Let $Q_{sH}=S_{sH}E_{sH}S_{sH}^{-1}$, and let $V_{sH}$ be a partial isometry for which
$V_{sH}^*V_{sH}=Q_{sH}$ and $V_{sH}V_{sH}^*=P_{sH}$.
\[
T=\sum_{sH\in G/H}P_{sH}V_{sH}S_{sH}E_{sH}\text{ and }
T'=\sum_{sH\in G/H}E_{sH}S_{sH}^{-1}V_{sH}^*P_{sH}
\]
where, again, convergence is in weak operator topology.  To see that these indeed converge, observe that
$T=\left(\sum_{sH\in G/H}P_{sH}V_{sH}S_{sH}S^{-1}P_{sH}\right)S$, and a similar formula holds for $T'$.
Notice, then, that
\[
\|T\|\leq \norm{\sum_{sH\in G/H}P_{sH}V_{sH}S_{sH}S^{-1}P_{sH}}\|S\|\leq \sup_{sH\in G/H}
\|S_{sH}\|\|S^{-1}\|\|S\|\leq \cbnorm{\pi}^3
\]
and, likewise, $\|T^{-1}\|\leq \cbnorm{\pi}^3$. For each $sH$ in $G/H$ we have  
$V_{sH}S_{sH}E_{sH}S_{sH}^{-1}V_{sH}^*=V_{sH}P_{sH}V_{sH}^*
=P_{sH}$, from which we deduce that $TT'=I$.  Similarly, $T'T=I$, so $T'=T^{-1}$.  
Furthermore, for $u$ in $\fal(G)$ we have
\begin{align*}
\pi_T(\bar{u})&=
\sum_{sH\in G/H}P_{sH}V_{sH}S_{sH}\pi(1_{sH}\bar{u})S_{sH}^{-1}V_{sH}^*P_{sH} \\
&=\sum_{sH\in G/H}P_{sH}V_{sH}\left[S_{sH}\pi(1_{sH}u)S_{sH}^{-1}\right]^*V_{sH}^*P_{sH}
=\pi_T(u)^*.
\end{align*}
Hence $T$ is the desired similarity operator.
\end{proof}

If $\pi:\fal(G)\to\fB(\fH)$ is a homomorphism, we may consider two ancilliary homomorphisms,
$\check{\pi},\pi^*:\fal(G)\to\fB(\fH)$ by $\check{\pi}(u)=\pi(\check{u})$ and
$\pi^*(u)=\pi(\bar{u})^*$.  Since $\fC\fB(\fal(G),\fB(\fH))\cong\vn(G)\bar{\otimes}\fB(\fH)$
(von Neuman tensor product, see, \cite[ 7.2.5 \&   7.2.4]{effrosrB}), 
if $\pi$ is completely bounded it may be identified with an operator $V_\pi$ in $\vn(G)\bar{\otimes}\fB(\fH)$,
which is called a {\it corepresentation}.  With this notation
the following is immediate from \cite[Theorems 9 \&16]{brannans}.

\begin{corollary}\label{cor:brannans}
Suppose that $G$  is either QSIN, or contains an amenable open normal subgroup.  Then for each 
completely bounded homomorphism $\pi:\fal(G)\to\fB(\fH)$ we have:

{\bf (i)}  $\check{\pi}$ and $\pi^*$ are completely bounded; and

{\bf (ii)} the associated corepresentation $V_\pi$ in $\vn(G)\bar{\otimes}\fB(\fH)$ is invertible
and hence similar to a unitary corepresentation.
\end{corollary}

In fact, it is shown in \cite{brannans} that for any locally compact $G$, either (i) or (ii) implies that 
$\pi$ is similar to a $*$-homomorphism.  A notable feature of their methods is that they
gain a proof,  \cite[Corollary 10]{brannans},
that the Gelfand spectrum of $\fal(G)$ is $G$, using Wendel's theorem characterizing
$\meas(G)$ as multipliers of $\bl^1(G)$.

\begin{conjecture}
For any locally compact group, we always have:
\[
\wtil{\fal}(G)_1\cong\fC_0(G)\text{ and }d_{\cb}(\fal(G))\leq 2.
\]
\end{conjecture}

\begin{remark}\label{rem:autobdd}
For a general amenable group $G$, it is unkown if any bounded homomorphism from $\fal(G)$
into $\fB(\fH)$ is automatically completely bounded, unless $G$ is virtually abelian; see \cite{forrestw}.
One method of attack would be to determine if
$\wtil{\max\fal}(G)_1=\fC_0(G)$, and, if that is true, then see if $m_{2,1}:\max\fal(G)\otimes^h\max\fal(G)
\to\fC_0(G)$ is completely surjective.  By a characterization of maximal operator spaces, the latter
property amounts  to verifying whether  there is a $K>0$ such that for each $n$, any
element $[h_{ij}]$ in the open unit ball of $M_n(\fC_0(G))\cong \fC_0(G,M_n)$ may be uniformly
approximated by elements of the form
\[
s\mapsto\left[\sum_{k,l=1}^m\alp_{ik}\beta_{kl}\gamma_{lj}u_kv_l(s)\right]=
\alp\mathrm{diag}(u_1(s),\dots,u_m(s))\beta\mathrm{diag}(v_1(s),\dots,v_m(s))\gamma
\]
where $m\in\En$, $\|\alp\|_{M_{n,m}}\|\beta\|_{M_m}\|\gamma\|_{M_{m,n}}<1$, and
$\|u_k\|_\fal<K$ and $\|v_k\|_\fal<K$ for each $k$.  Indeed, this is the Blecher-Paulsen
factorization of \cite{pisier}.
\end{remark}

\section{The case of similarity degree $1$}

We close by addressing the case when $d_{\cb}(\fal(G))=1$. 

\begin{proposition}\label{prop:simone}
For a locally compact group $G$ the following are equivalent:

{\bf (i)} $G$ is finite;

{\bf (ii)} $d_{\cb}(\fal(G))=1$; and

{\bf (iii)} $\fal(G)$ is completely isomorphic to an operator algebra.
\end{proposition}

\begin{proof}
If $G$ is finite, then we know that the inclusion map $j:\fal(G) \to \fC_0(G)$ is a complete isomorphism, which factors through $\tilde{\fal}(G)_1$. This means that $\iota_1 : \fal(G) \to \wtil{\fal}(G)_1$ is also a complete surjection (in fact, a complete isomorphism), which implies that $d_{cb}(\fal(G)) =1$.
Hence (i) implies (ii) implies (iii).

Let us suppose that $\fal(G)$ is completely isomorphic to an operator algebra, 
which implies that it is Arens regular (\cite{civiny}).
By a result of Forrest (\cite[Theorem 3.2]{forrest}), the Arens regularity tells us that $G$
is discrete.  We complete the proof in two distinct ways.   

(I)  It follows from Theorem \ref{theo:main} (or from \cite[Theorem 20]{brannans})
that $\fal(G)$ is isomorphic to a C*-algebra, hence has a bounded approximate identity.
Thus, by Leptin's theorem \cite{leptin}, $G$ is amenable.  But then, \cite[Corollary 3.6]{forrest}
tells us that if $G$ is amenable and $\fal(G)$ is Arens regular, then $G$ is finite.

(II)  For a discrete group $G$, the result \cite[Theorem 5.4]{leey}, of the first named author and Youn,
implies that $\fal(G)$ is completely isomorphic to an operator algebra only if the constant function
$1$ is square integrable, hence $G$ is finite.

In summary, (iii) implies (i).
\end{proof}

We note that the full characterization of those $G$ for which $\fal(G)$ is Arens regular remains unsolved.
To the authors' knowledge, the strongest result at present is that if $\fal(G)$ is Arens
regular, then $G$ is a discrete group containing no infinite amenable
subgroups; see \cite{forrest,forrest1}.  Likewise, we do not know of a full characterization of those $G$ 
for which $\fal(G)$ is boundedly isomorphic to an operator algebra.

Combining Theorem \ref{theo:main} and the above proposition yields the following.

\begin{corollary}
For an infinite QSIN group $G$, $d_{\cb}(\fal(G))=2$.
\end{corollary}

\section{Acknowledgements}
The authors benefitted from conversations with Y. Choi, who was aware of the techniques used in
Remark \ref{rem:relations} (ii).  The third author benefitted from conversations with V.I. Paulsen, from 
whom he learned the significance of the factorization used in Remark \ref{rem:autobdd}; and from
conversations with J. Crann, who kindly pointed out some errors in an earlier version of 
Remark \ref{rem:scope}.  The authors are also grateful to the anonymous referee who read the 
paper carefully and suggested many improvements to the exposition.


Addresses:

\noindent {\sc 
Department of Mathematical Sciences  and Research Institute of Mathematics, Seoul National University,
Gwanak-ro 1, Gwanak-gu, Seoul 08826, Republic of Korea \\ \\
Department of Mathematics and Statistics, University of Saskatchewan,
Room 142 McLean Hall, 106 Wiggins Road
Saskatoon, sk, S7N 5E6  \\ \\
Department of Pure Mathematics, University of Waterloo,
Waterloo, ON, N2L 3G1, Canada.}

\medskip
Email-adresses:
\linebreak
{\tt hunheelee@snu.ac.kr}
\linebreak {\tt samei@math.usask.ca}
\linebreak {\tt nspronk@uwaterloo.ca}

\end{document}